\documentclass[reqno]{amsart}
\usepackage{amsmath,amssymb}
\usepackage{hyperref}

\newcommand*{\mailto}[1]{\href{mailto:#1}{\nolinkurl{#1}}}

\allowdisplaybreaks 

\newtheorem{theorem}{Theorem}[section]
\newtheorem{proposition}[theorem]{Proposition}

\newtheorem{lemma}[theorem]{Lemma}

\theoremstyle{definition}

\newtheorem{remark}[theorem]{Remark}
\newtheorem{hypothesis}[theorem]{Hypothesis}

\numberwithin{equation}{section}

\begin{document}

\title[Lower bounds for self-adjoint Sturm--Liouville operators]{Lower bounds for self-adjoint Sturm--Liouville operators}

\author[J. Behrndt]{Jussi Behrndt}
\address{Technische Universit\"{a}t Graz\\
Institut f\"ur Angewandte Mathematik\\
Steyrergasse 30\\
8010 Graz, Austria}
\email{\mailto{behrndt@tugraz.at}}
\urladdr{\url{https://www.math.tugraz.at/~behrndt/}}

\author[F.\ Gesztesy]{Fritz Gesztesy}
\address{Department of Mathematics,
Baylor University, Sid Richardson Bldg., 1410 S.\,4th Street,
Waco, TX 76706, USA}
\email{\mailto{Fritz\_Gesztesy@baylor.edu}}
\urladdr{\url{http://www.baylor.edu/math/index.php?id=935340}}

\author[P. Schmitz]{Philipp Schmitz}
\address{Department of Mathematics\\
	Technische Universit\"at Ilmenau\\ Postfach 100565\\
	98648 Ilmenau\\ Germany}
\email{\mailto{philipp.schmitz@tu-ilmenau.de}}
\urladdr{\url{https://www.tu-ilmenau.de/obc/team/philipp-schmitz}}

\author[C. Trunk]{Carsten Trunk}
\address{Department of Mathematics\\
Technische Universit\"at Ilmenau\\ Postfach 100565\\
98648 Ilmenau\\ Germany}
\email{\mailto{carsten.trunk@tu-ilmenau.de}}
\urladdr{\url{https://www.tu-ilmenau.de/funktionalanalysis}}


\thanks{Appeared in {\it Proc. Amer. Math. Soc.} {\bf 151}, 5313--5323 (2023).}

\begin{abstract}
In this note we provide estimates for the lower bound of the self-adjoint operator associated with the three-coefficient Sturm--Liouville differential expression
$$
\frac{1}{r} \left(-\frac{\mathrm d}{\mathrm dx} p \frac{\mathrm d}{\mathrm dx} + q\right)
$$
in the weighted $L^2$-Hilbert space $L^2(\mathbb R; rdx)$.
\end{abstract}

\maketitle

\section{Introduction}

One-dimensional Schr\"{o}dinger operators of the form
\begin{equation}\label{aaa}
H=-\frac{\mathrm d^2}{\mathrm dx^2}+ q
\end{equation}
with a real-valued potential $q$
have been studied in the mathematical and physical literature intensively in the last century due to their particular importance in quantum mechanics. Typically one is interested in a suitable self-adjoint 
realization in $L^2(\mathbb R)$ and its spectral properties, among them 
estimates for lower bounds, numbers of negative eigenvalues, and Lieb--Thirring inequalities are particularly
important, see, for instance, the recent survey \cite{Fr21}. 

The main objective of this note is to derive estimates on the lower bound of more general Sturm--Liouville operators of the type
\begin{equation}
	T = \frac{1}{r} \left(-\frac{\mathrm d}{\mathrm dx} p \frac{\mathrm d}{\mathrm dx} + q\right)
\end{equation} 
with real-valued coefficients under the standard assumptions $r,1/p,q\in L^1_{\rm loc}(\mathbb R)$ and $r,p$ positive almost everywhere.
We refer the reader to the textbooks \cite{DS88}, \cite{GNZ23}, \cite{JR76}, \cite{Pe88}, \cite{Sc81}, \cite{Ti62}, \cite{We03}, and \cite{Ze05} for an overview and detailed study of Sturm--Liouville (resp., Schr\"odinger) operators. 
The natural Hilbert space in this context is the weighted $L^2$-space  $L^2_r(\mathbb R):=L^2(\mathbb R; rdx)$ and 
under some mild additional assumptions on the coefficients one concludes that $T$
is a semibounded self-adjoint operator in $L^2_r(\mathbb R)$.
As mentioned above, lower bounds for the spectrum of $T$ 
are known for the special case $r=p=1$, that is, $T=H$, and for completeness we provide a straightforward estimate as
a warm up in Section~\ref{warm}. 

In the general setting
it seems that a systematic study is missing and it is the aim of this note to initiate and contribute to this circle of problems. It is clear that 
the coefficients $r$ and $p$ have an essential influence on the lower bound. If, for instance, the weight function $r=r_0$ is constant and $p=1$ then formally
$T=(1/r_0) H$ and the lower bound $\inf(\sigma(T))$ of $T$ is simply given by $(1/r_0) \inf(\sigma(H))$. This already indicates that for a nonconstant weight function $r$ the $L^\infty$-norm of 
$1/r$ will appear in the lower bounds, and the situation becomes much more difficult if $1/r\not \in L^\infty(\mathbb R)$, in which case we require the existence of a function $g$ that
neutralizes the behaviour of the weight function $r$ on subsets of $\mathbb R$ where $r$ is small.
Furthermore, the norm of the 
coefficient $p$ will enter in lower bound estimates and very roughly speaking $1/p$ has to be considered in conjunction with the potential $q$. The methods and proofs in this paper are strongly inspired by 
\cite{BST19}, where bounds on nonreal eigenvalues of indefinite Sturm--Liouville operators are obtained.

\subsection*{Acknowledgements.} 
Jussi Behrndt gratefully acknowledges financial support by the Austrian Science Fund (FWF): P 33568-N.
The authors wish to thank the Erwin Schr\"odinger International Institute for Mathematics and Physics (ESI) at 
University of Vienna, where this article was finalized during the workshop \textit{Spectral Theory of Differential Operators} in November 2022.
This publication is also based upon work from COST Action CA 18232
MAT-DYN-NET, supported by COST (European Cooperation in Science and
Technology), www.cost.eu.

\section{One-dimensional Schr\"{o}dinger operators}\label{warm}

As a warm up we discuss in this short section the special case $p=r=1$ and $q\in L^s(\mathbb R)$ real-valued~a.e., $s \in [1,\infty]$, and derive a lower bound for the self-adjoint Schr\"{o}dinger operator $H$ in \eqref{aaa} using the argument presented in \cite[(3.5.30), p.~155--156]{Th3}. 

We start by recalling that $q \in L^s(\mathbb R)$, $s \in [1,\infty)$, implies that $q$ is relatively form compact with respect to the free Hamiltonian $H_0$ in $L^2(\mathbb R)$, where
\begin{equation}
H_0 f = - f'', \quad f \in D(H_0) = H^2(\mathbb R),
\end{equation}
with $H^{\ell}(\mathbb R)$, $\ell \in [0,\infty)$, the standard scale of Sobolev spaces. This follows from the stronger statement that $|q|^{1/2} (H_0 + I)^{-1/2}$ satisfies (see, e.g., \cite[Theorem~XI.20]{RS79})
\begin{equation}
|q|^{1/2} (H_0 + I)^{-1/2} \in \mathcal{B}_{2s}\bigl(L^2(\mathbb R)\bigr) \, \text{ if } \, q \in L^s(\mathbb R), \; s \in [1,\infty),
\end{equation}
where $\mathcal{B}_t(\mathcal{H})$ represent the $\ell^t(\mathbb N)$-based 
trace ideals of compact operators in the complex, separable Hilbert space $\mathcal H$. In particular, 
\begin{equation}
|q|^{1/2} (H_0 + I)^{-1/2} \, \text{ is compact,}   \label{2.3} 
\end{equation} 
and hence the form sum $H$ of $H_0$ and $q$ is self-adjoint in $L^2(\mathbb R)$ and bounded from below. By a result of Hartman \cite{Ha48} and Rellich \cite{Re51} 
(see also \cite[Theorem~8.5.2]{GNZ23}), the boundedness from below of the minimal operator associated with the differential expression $- (d^2/dx^2) + q$ implies that the latter is in the limit point case at $\pm \infty$ and hence the maximal operator associated with $- (d^2/dx^2) + q$ is self-adjoint in $L^2(\mathbb R)$, and thus necessarily coincides with $H$. 
It is clear that for $s=\infty$ the same is true as $q \in L^\infty(\mathbb R)$ is a bounded perturbation of $H_0$.
Consequently, $H$ is given by
\begin{equation*}
\begin{split}
& H f = -f'' + qf,  \\
& f \in D(H) = \bigl\{g\in L^2(\mathbb R) \,\big|\, g, g' \in AC_{\rm loc}(\mathbb R); \, (- g'' + qg) \in L^2(\mathbb R) \bigr\}.
\end{split}
\end{equation*}

Property \eqref{2.3} then implies
\begin{equation} 
\sigma_{\rm ess}(H) = \sigma_{\rm ess}(H_0) = [0,\infty),  
\end{equation}
and hence it suffices to consider negative eigenvalues,
which turn out to be simple as $-(d^2/dx^2) +q$ is in the limit point case at $\pm \infty$. 
We consider an eigenvalue
$\lambda<0$ of $H$ and denote the corresponding eigenfunction 
by $f_\lambda$. From $-f_\lambda''+qf_\lambda=\lambda f_\lambda$ one concludes  
\begin{equation*}
 f_\lambda=-\left(-\frac{\mathrm d^2}{\mathrm dx^2}-\lambda\right)^{-1}qf_\lambda,
\end{equation*}
extending $\big(- (d^2/dx^2) - \lambda\big)^{-1}$ from $L^2(\mathbb{R})$ to $L^{t'}(\mathbb{R})$, $t' \in [1,2]$ (cf.\ \eqref{2.1}).

Using the corresponding Green's function we obtain 
\begin{equation} \label{2.1}
 \Vert f_\lambda\Vert_r=\frac{1}{2\sqrt{-\lambda}}\bigl\Vert e^{-\sqrt{-\lambda}\vert \, \cdot \, \vert} \ast qf_\lambda 
 \bigr\Vert_r \leq \frac{1}{2\sqrt{-\lambda}} \bigl\Vert e^{-\sqrt{-\lambda}\vert \, \cdot \, \vert} \bigr\Vert_t
 \Vert qf_\lambda\Vert_{t'},
\end{equation}
where Young's inequality,\footnote{Explicitly, $\|f \ast g\|_{\alpha} \leq \|f\|_{\beta} \|g\|_{\gamma}$, $1 \leq \alpha, \beta,\gamma \leq \infty$, $1 + \alpha^{-1} = \beta^{-1} + \gamma^{-1}$.} with $1/t+1/t'=1+(1/r)$ was applied in the last step. H\"older's inequality then yields 
$\Vert qf_\lambda\Vert_{t'}\leq \Vert q\Vert_s \Vert f_\lambda\Vert_r$ for $1/t'=1/s+(1/r)$ and hence,  
\begin{equation} \label{2.2}
 \sqrt{-\lambda}\leq \frac{1}{2} \bigl\Vert e^{-\sqrt{-\lambda}\vert \, \cdot \, \vert} \bigr\Vert_t
 \Vert q\Vert_s=\frac{1}{2}\left(\frac{2}{t\sqrt{-\lambda}}\right)^\frac{1}{t}\Vert q\Vert_s
\end{equation}
if $t \in (1,\infty)$, that is, if $s \in (1,\infty)$. As $1/t=1-(1/s)$ it follows for $s \in (1,\infty)$ that
\begin{equation*}
 (-\lambda)^\frac{2s-1}{2s}\leq 2^{-\frac{1}{s}}\left(\frac{s-1}{s}\right)^\frac{s-1}{s}\Vert q\Vert_s 
\end{equation*}
and hence 
\begin{equation}\label{warmupbound}
 \inf(\sigma(H)) \geq -2^{-\frac{2}{2s-1}}\left(\frac{s-1}{s}\right)^\frac{2(s-1)}{2s-1}\Vert q\Vert_s^\frac{2s}{2s-1} 
 \, {\text{ for } \, s \in (1,\infty)}.     
\end{equation}
It is easy to see that the lower bound \eqref{warmupbound} also remains valid for $s=1$ (indeed, inequality \eqref{2.1} and the first inequality in \eqref{2.2} apply with $s=1$, $r = t = \infty$, $t' = 1$) in which case one obtains 
\begin{equation} \label{2.8}
\inf(\sigma(H)) \geq - (1/4) \Vert q\Vert_1^2, 
\end{equation} 
and obviously applies to $s = \infty$ in the form 
\begin{equation} \label{2.9}
\inf(\sigma(H)) \geq - \Vert q\Vert_{\infty}.
\end{equation} 
We mention that the bounds \eqref{warmupbound} and \eqref{2.8} coincide with 
\cite[Corollary~14.3.11 and Corollary 14.3.12]{Da07} and that the above argument also leads to bounds for 
Schr\"{o}dinger operators with complex potentials $q\in L^s(\mathbb R)$, $s \in [1,\infty)$ (see, in particular, \cite{AAD01} for the case $s=1$). In the context of Schr\"odinger operators with complex-valued potentials we also refer, for instance, to  \cite{CI20}, \cite{Fr11}, \cite{Fr18}, and \cite{FLLS06}.

\begin{remark}
We mention that Lieb--Thirring inequalities (see, e.g., \cite{Fr21}, \cite{La12} and the extensive literature cited therein) also lead to lower bounds of $H$. More specifically,
for $s=3/2$ one can compare with the one-particle constant $L_1^{(1)} = 4\big/\big[3^{3/2} \pi\big] = 0.24503$ in \cite[Section 3.2]{Fr21}: The corresponding constant in \eqref{warmupbound} equals $2^{-1} 3^{-1/2} = 0.28867$.  
Historically, we note that Barnes, Brascamp, and Lieb \cite{BBL76} derived a lower bound for the ground state energy of (multi-dimensional) Schr\"odinger operators already in 1976.
\end{remark}

\section{Main results}
Now we consider the general Sturm--Liouville differential expression
\begin{equation}
	\tau = \frac{1}{r} \left(-\frac{\mathrm d}{\mathrm dx} p \frac{\mathrm d}{\mathrm dx} + q\right)
\end{equation} on $\mathbb R$ with real-valued coefficients under the standard assumptions  
\begin{equation}
r,1/p,q\in L^1_{\rm loc}(\mathbb R), 
\end{equation} 
and we assume from now on that Hypothesis~\ref{hypo_dwig} below is satisfied. 
In the following $L^1_\mathrm u(\mathbb R)$ denotes the normed space of uniformly locally integrable functions, that is,
\begin{equation*}
	L^1_{\mathrm u}(\mathbb R) = \big\{h\in L^1_{\mathrm{loc}}(\mathbb R) : \lVert h\rVert_{\mathrm{u}} <\infty\big\},\qquad \lVert h\rVert _{\mathrm u} = \sup_{n\in\mathbb Z} \int_{n}^{n+1} \lvert h(t) \rvert\,\mathrm d t.
\end{equation*}

\begin{hypothesis} \label{h3.1} The real coefficients $p$, $q$ and $r$ of $\tau$ satisfy the following:
	\label{hypo_dwig}
	\begin{enumerate}
		\item[(a)] $p(x)>0$ for a.a.\ $x\in\mathbb R$ and $1/p\in L^\eta(\mathbb R)$ for some $\eta\in [1,\infty]$;
		\item[(b)] $q\in L^1_{\mathrm u}(\mathbb R)$;
		\item[(c)] $r(x)>0$ for a.a.\ $x\in\mathbb R$ and there exist $a$,~$b\in\mathbb R$ with $a<b$ such that 
		\begin{equation}
			\operatorname{ess\,inf}_{t\in \mathbb R\setminus [a,b]} r(t)>0.		\qedhere
		\end{equation}
	\end{enumerate}
\end{hypothesis}
It is known that the differential expression $\tau$ is in the limit-point case at both singular endpoints $\pm\infty$, and the corresponding 
maximal operator
\begin{equation*}
\begin{split}
& Tf = \tau f= \frac{1}{r} \left(-(p f')'  + qf\right),\\
& f \in D(T) = \bigl\{g\in L^2_r(\mathbb R) \,\big|\, g, pg' \in AC_{\rm loc}(\mathbb R); \, \tau g\in L^2_r(\mathbb R) \bigr\},
\end{split}
\end{equation*}
is self-adjoint in the weighted $L^2$-space $L^2_r(\mathbb R)$ and semibounded from below; cf. \cite[Lemma~A.2]{BST19}. Our main goal is to derive estimates for the lower bound $\inf(\sigma(T))$ of $T$.
For a nonnegative function $g\in L^\infty(\mathbb R)$ the set 
 	$$\Omega_g:=\{x\in\mathbb R \,|\, r(x)g(x)<1\}$$
 	and its Lebesgue measure $\mu(\Omega_g)$ will appear in the lower bound estimates
 	in our main results below.
 	In the particular case $1/r\in L^\infty(\mathbb R)$ one can choose $g=1/r$ 
 	and hence $\Omega_g$ becomes a Lebesgue null set, which leads to more explicit lower bound estimates. We also mention that for any $\varepsilon>0$ there exists a (constant)
 	nonnegative function $g_\varepsilon\in L^\infty(\mathbb R)$ such that $\mu(\Omega_{g_\varepsilon})<\varepsilon$; this follows from
 	\begin{equation*}
	\begin{split}
		\lim_{n\rightarrow \infty} \mu(\{x\in\mathbb R \,|\, r(x)<1/n\}) &= \mu \left(\bigcap_{n=1}^\infty \{x\in\mathbb R \,|\, r(x)<1/n\}\right)\\
		&= \mu(\{x\in\mathbb R \,|\, r(x)=0\})=0.
		\end{split}
	\end{equation*}

 	The first main result requires only minimal assumptions on the potential in 
 	Hypothesis~\ref{hypo_dwig}, that is, 
 	$q\in L^1_\mathrm u(\mathbb R)$, but we have to assume that $1/p\in L^\infty(\mathbb R)$.
We shall denote the negative part of the 
 	potential $q$ by $q_-$.
 	
\begin{theorem}\label{thm1}
In addition to Hypothesis \ref{h3.1}, assume that $1/p\in L^\infty(\mathbb R)$, let 
 	\begin{equation}
			      \label{erdnuss1}
			      \alpha= 2\lVert q_- \Vert_{\mathrm u} +4  \lVert 1/p \rVert_\infty \lVert q_-\Vert_{\mathrm u}^2\quad \text{and}
			      \quad \beta=\left(4\lVert 1/p \rVert_{\infty} \alpha\right)^{1/2},
		      \end{equation}
 	and choose a nonnegative function $g\in L^\infty(\mathbb R)$ such that $\mu(\Omega_g)\beta<1$.
	Then 
	\begin{equation*}
		\inf(\sigma(T)) \geq \frac{-\alpha\lVert g \rVert_\infty}{1-\mu(\Omega_g)\beta}
	\end{equation*}
	and in the special case $1/r\in L^\infty(\mathbb R)$ the choice $g=1/r$ implies 
	$\mu(\Omega_g)=0$ and
	 $$
		\inf(\sigma(T)) \geq -\bigl(2\lVert q_- \Vert_{\mathrm u} +4  \lVert 1/p \rVert_\infty \lVert q_-\Vert_{\mathrm u}^2\bigr)\lVert 1/r \rVert_\infty.
	$$
\end{theorem}

In the next result we consider the case $q_-\in L^s(\mathbb R)$, $s\in [1,\infty]$, and $1/p\in L^\infty(\mathbb R)$.

\begin{theorem}\label{thm2}
In addition to Hypothesis \ref{h3.1}, assume that $1/p\in L^\infty(\mathbb R)$ and $q_-\in L^s(\mathbb R)$ for some $s\in [1,\infty]$,
 	let 
 	\begin{equation}
			      \label{erdnuss10}
			      \alpha=\begin{cases}  
						\Vert q_- \Vert_s \beta^{\frac{1}{s}} & \text{if}\,\,s\in[1,\infty), \\[1ex]
						\Vert q_- \Vert_\infty & \text{if}\,\,s=\infty,\end{cases}  
			      \quad \text{and}
			      \quad \beta=\begin{cases} \left
						(4\lVert 1/p \rVert_\infty \lVert q_- \Vert_{s}\right)^{\frac{s}{2s-1}}  & \text{if}\,\,s\in[1,\infty), \\[1ex] 
						(4\lVert 1/p \rVert_\infty \lVert q_- \Vert_{\infty})^{1/2}& \text{if}\,\,s=\infty,\end{cases} 
		      \end{equation}
 	and choose a nonnegative function $g\in L^\infty(\mathbb R)$ such that $\mu(\Omega_g)\beta<1$.
	Then 
	\begin{equation*}
		\inf(\sigma(T)) \geq \frac{-\alpha\lVert g \rVert_\infty}{1-\mu(\Omega_g)\beta}
	\end{equation*}
	and in the special case $1/r\in L^\infty(\mathbb R)$ the choice $g=1/r$ implies 
	$\mu(\Omega_g)=0$ and
	\begin{equation}\label{ojemine}
	\inf(\sigma(T)) \geq  - \bigl(4\lVert 1/p \rVert_\infty\bigr)^{\frac{1}{2s-1}} \lVert q_- \Vert_{s}^{\frac{2s}{2s-1}}\lVert 1/r \rVert_\infty   \quad \text{if } s\in[1,\infty).
	\end{equation}
	If $s=\infty$ we have
	$$
	\inf(\sigma(T)) \geq -\Vert q_- \Vert_\infty \lVert 1/r \rVert_\infty.
	$$
\end{theorem}

\begin{remark}\label{remmi}
The bounds in Theorem~\ref{thm2} above are not optimal. In fact, in the special case $p=r=1$ and $q_-\in L^s(\mathbb R)$ for some $s\in [1,\infty)$, the bound in \eqref{ojemine} becomes 
\begin{equation*}
	\inf(\sigma(T)) \geq \begin{cases} -4   \lVert q_- \Vert_{1}^2 & \text{if}\,\,s=1, \\[1ex] 
	- 2^{\frac{2}{2s-1}} \lVert q_- \Vert_{s}^{\frac{2s}{2s-1}}  & \text{if}\,\,s\in(1,\infty),\end{cases}
	\end{equation*}
while  \eqref{warmupbound} (or \cite[Corollary~14.3.11 and Corollary 14.3.12]{Da07}) show  that 
	\begin{equation}
		\inf(\sigma(T)) \geq \begin{cases} - \lVert q\rVert_1^2\big/4 & \text{if}\,\, s=1,\\[1ex] 
			- 2^{-\frac{2}{2s-1}} \left(\frac{s-1}{s}\right)^{\frac{2(s-1)}{2s-1}} \lVert q\rVert_s^{\frac{2s}{2s-1}} & \text{if}\,\, s\in (1,\infty).
		\end{cases}
	\end{equation}
\hfill $\diamond$	
\end{remark}

In the following theorem we deal with $q_-\in L^s(\mathbb R)$, $s\in [1,\infty]$, and $1/p\in L^\eta(\mathbb R)$, $\eta\in [1,\infty)$.

\begin{theorem}\label{thm3} 
In addition to Hypothesis \ref{h3.1}, assume that $1/p\in L^\eta(\mathbb R)$ for some $\eta\in [1,\infty)$ and $q_-\in L^s(\mathbb R)$ for some $s\in [1,\infty]$ such that $\eta+s >2$ if $s\not=\infty$. 
 	Let $\alpha$ be as in \eqref{erdnuss10} and
	\begin{equation}\label{olala}
	\beta=\begin{cases} \Bigl(\Bigl(\frac{2\eta-1}{\eta}\Bigr)^2\lVert 1/p \rVert_\eta \lVert q_- \Vert_s\Bigr)^{\frac{\eta s}{2\eta s -\eta -s}}  & \text{if}\,\,s\in[1,\infty), \\[1ex] 
		\Bigl(\Bigl(\frac{2\eta-1}{\eta}\Bigr)^2 \lVert 1/p \rVert_\eta \lVert q_- \Vert_\infty\Bigr)^{\frac{\eta}{2\eta-1}} & \text{if}\,\,s=\infty,\end{cases} 
		      \end{equation}
 	and choose a nonnegative function $g\in L^\infty(\mathbb R)$ such that $\mu(\Omega_g)\beta<1$.
	Then 
	\begin{equation*}
		\inf(\sigma(T)) \geq \frac{-\alpha\lVert g \rVert_\infty}{1-\mu(\Omega_g)\beta}
	\end{equation*}
	and in the special case $1/r\in L^\infty(\mathbb R)$ the choice $g=1/r$ implies 
	$\mu(\Omega_g)=0$ and
	\begin{equation*}
	\inf(\sigma(T)) \geq
	-\Bigl(\Bigl(\frac{2\eta-1}{\eta}\Bigr)^2\lVert 1/p \rVert_\eta \Bigr)^{\frac{\eta}{2\eta s -\eta -s}} 
	\lVert q_- \Vert_{s}^{\frac{2\eta s-s}{2\eta s -\eta -s}} \lVert 1/r \rVert_\infty  \quad\text{if } s\in[1,\infty).
	\end{equation*}
	If $s=\infty$ we have
	$$
	\inf(\sigma(T)) \geq-\Vert q_- \Vert_\infty \lVert 1/r \rVert_\infty. 
	$$
\end{theorem}

The following proposition provides a simple condition for nonnegativity of $T$.

\begin{proposition}\label{propi}
 	Suppose that Hypothesis~\ref{hypo_dwig} holds and assume that $1/p,q_-\in L^1(\mathbb R)$. If $\lVert 1/p \rVert_1\lVert q_- \Vert_1<1$ 
 	then $\inf(\sigma(T)) \geq 0$.
 \end{proposition}	

\begin{remark}
 If the coefficients $p$, $q$ and $r$ in Hypothesis~\ref{hypo_dwig} are restricted to the half line $(0,\infty)$ and $T_+$ denotes the 
 self-adjoint realization of the (restricted) differential expression $\tau$ in $L^2((0,\infty))$ with Dirichlet boundary conditions at the regular endpoint $0$, 
 then the lower bound estimates above remain valid for $T_+$. In fact, the Dirichlet boundary condition at $0$ ensures that the boundary term in the integration by parts formula for $f\in D(T_+)$ vanishes and hence the proofs in the next section (see, e.g. \eqref{rotbarschfilet}) extend directly to the half line case.
 \end{remark}

 \section{Proofs}

 In this section we prove our main results. It will always be assumed that the coefficients $p,q,r$ satisfy Hypothesis~\ref{hypo_dwig}.
The first three items of the following useful statement can be found, for instance, in \cite[Lemma~A.2]{BST19}. The 
last item follows from \cite[Lemma A.1]{BST19}
	and the first item.

\begin{lemma}
	\label{krolik}
Assume Hypothesis \ref{h3.1}, then the following assertions hold for all $f$,~$g\in D(T)$:
	\begin{itemize}
		\item[(i)] $f$,~$\sqrt pf'\in L^2(\mathbb R)$ and $qf^2\in L^1(\mathbb R)$;
		\item[(ii)]there exists a sequence $(x_n)_{n\in\mathbb Z}$ in $\mathbb R$ satisfying $\lim_{n\rightarrow \infty}x_n=\infty$ and $\lim_{n\rightarrow -\infty}x_n=-\infty$ such that $\lim_{\lvert n \rvert\rightarrow \infty}f(x_n)= 0$;
		\item[(iii)] $\lim_{\lvert x \rvert\rightarrow \infty} (pf')(x)\overline{g(x)}=0$;
		\item [(iv)] $f \in L^\infty(\mathbb R)$.
	\end{itemize}
\end{lemma}

For our estimates of the lower bound of $T$ it is convenient to reduce the considerations to 
the set
\begin{equation}
	\label{dminus}
	D_-(T) := \{f\in D(T) \,|\, (T f, f)_r\leq 0\},
\end{equation}
where $(\cdot ,\cdot)_r$ stands for the weighted inner product 
corresponding to $L_r^2(\mathbb R)$.
The potential $q$ is decomposed in its positive part $q_+$ and negative part $q_-$, i.\,e.
\begin{equation}
	q=q_+ - q_-,\quad\text{where}\quad q_+ := \frac{\lvert q \rvert + q}{2}\quad\text{and}\quad q_- := \frac{\lvert q \rvert - q}{2}.
\end{equation}

\begin{lemma}
	\label{A.6}
Assuming Hypothesis \ref{h3.1}, every function $f\in D_-(T)$ satisfies
	\begin{equation}
		\lVert \sqrt{p}f' \rVert_2^2 \leq \lVert q_-f^2 \rVert_1\quad\text{and}\quad\lVert qf^2 \rVert_1 \leq 2 \lVert q_-f^2 \rVert_1.
	\end{equation}
	Moreover, the inequality $\lVert q_- f^2 \lVert_1 \leq \lVert q_+ f^2 \Vert_1$ implies $\lVert \sqrt{p}f' \rVert_2 =0$.
\end{lemma}

\begin{proof}
	For $f\in D_-(T)$ integration by parts together with Lemma~\ref{krolik}~(i) and (iii) yields
	\begin{equation}
		\label{rotbarschfilet}
		\begin{split}
		0 \geq (Tf,f)_r & = \int_{\mathbb R} p(t)\lvert f'(t) \rvert^2\,\mathrm d t + \int_{\mathbb R} q(t)\lvert f(t) \rvert^2\,\mathrm d t\\
		& = \lVert \sqrt{p}f' \rVert_2^2 + \lVert q_+f^2 \rVert_1 - \lVert q_- f^2 \rVert_1.
	\end{split}
	\end{equation}
	This implies $\lVert \sqrt{p}f' \rVert_2^2 \leq \lVert q_-f^2 \rVert_1$ and $\lVert q_+f^2 \rVert_1\leq \lVert q_-f^2 \rVert_1$. Therefore, with $\lvert q \rvert=q_++q_-$ we have
	\begin{equation*}
		\lVert qf^2 \rVert_1 = \lVert q_+f^2 \rVert_1 + \lVert q_-f^2 \rVert_1 \leq 2\lVert q_-f^2 \rVert_1.%\qedhere
	\end{equation*}
	If $\lVert q_- f^2 \rVert_1 \leq \lVert q_+f^2 \rVert_1$ holds, then \eqref{rotbarschfilet} implies $\lVert \sqrt{p}f' \rVert_2 =0$.
\end{proof}

\begin{lemma}
	\label{elfo}
In addition to Hypothesis \ref{h3.1}, assume that there are constants $\alpha\geq 0$, $\beta\geq 0$ and a nonnegative function $g\in L^\infty(\mathbb R)$ such that
	\begin{enumerate}
	    \item $\lVert q_- f^2 \rVert_{1}\leq \alpha \Vert f \rVert_{2}^2$ and $\Vert f \rVert_\infty^2\leq \beta \Vert f \rVert_{2}^2$ 
		       for all $f\in D_-(T)$;
		\item $\mu(\Omega_g)\beta<1$. 
	\end{enumerate}
	Then one has
	\begin{equation}
		\label{infspecTmax}
		\inf(\sigma(T)) \geq \frac{-\alpha\lVert g \rVert_\infty}{1-\mu(\Omega_g)\beta}.
	\end{equation}
\end{lemma}

\begin{proof}
	Let $f\in D_-(T)$. 
	 Then one has 
\begin{equation}\label{kaugummi}
\begin{split}
\lVert g \rVert_\infty (f,f)_r&=\lVert g \rVert_\infty\int_{\mathbb R} \lvert f(t) \rvert^2 r(t)\,\mathrm d t\geq\int_{\mathbb R} \lvert f(t) \rvert^2 r(t)g(t)\,\mathrm d t     \\[0.5\baselineskip]
&\geq \int_{\mathbb R\setminus\Omega_g} \lvert f(t) \rvert^2 r(t)g(t)\,\mathrm d t \geq \Vert f  
\rVert_2^2 - \int_{\Omega_g} \lvert f(t) \rvert^2\,\mathrm d t      \\[0.5\baselineskip]
&\geq \Vert f \rVert_2^2 - \mu(\Omega_g) \Vert f \rVert_\infty^2 \geq \bigl(1 - \mu(\Omega_g) 
\beta\bigr) \Vert f \rVert_2^2.  
\end{split}
\end{equation}
Further, we have by \eqref{rotbarschfilet}
\begin{equation*}
		(T f,f)_r =  \lVert \sqrt{p}f' \rVert_2^2 + \lVert q_+f^2 \rVert_1 - \lVert q_- f^2 \rVert_1 \geq - \lVert q_-f^2 \rVert_1 \geq -\alpha \Vert f \rVert_2^2.
	\end{equation*}
	This together with \eqref{kaugummi} yields
	\begin{equation}
		\label{lausitz}
		(T f,f)_r  \geq -\frac{\alpha \lVert g \rVert_\infty}{1 - \mu(\Omega_g)\beta}(f,f)_r.
	\end{equation}
	Obviously, the inequality in \eqref{lausitz} holds also for $f\in D(T)\setminus D_-(T)$ and, thus, for all $f\in D(T)$. This implies \eqref{infspecTmax}
\end{proof}

Next we recall estimates on the $L^\infty$-norm of functions in $D(T)$ from \cite{BST19}. 

\begin{lemma}
	\label{Chewbacca} 
Assume  Hypothesis \ref{h3.1}. Then the following assertions hold for all $f\in D(T)$.
	\begin{enumerate}
		\item If $1/p\in L^\eta(\mathbb R)$, where $\eta\in [1,\infty)$, then
		      \begin{equation}
			      \label{Bischleben}
			      \Vert f \rVert_{\infty}  \leq \left(\frac{2\eta-1}{\eta} \sqrt{\lVert 1/p \rVert_{\eta}} \lVert \sqrt p f' \rVert_2\right)^{\frac{\eta}{2\eta-1}} \Vert f \rVert_2^{\frac{\eta-1}{2\eta-1}}.
		      \end{equation}
		\item If $1/p\in L^\infty(\mathbb R)$ then
		      \begin{equation}
			      \label{50Watt}
			      \Vert f \rVert_\infty \leq \left(2\sqrt{\lVert 1/p \rVert_\infty}\lVert \sqrt{p}f' \rVert_2 \Vert f \rVert_2\right)^{1/2}.
		      \end{equation}
		      Moreover, for every $\varepsilon>0$ and all $n\in\mathbb Z$ one has
		      \begin{equation}
			      \label{pullover}
			      \sup_{t\in [n,n+1]} \lvert f(t) \rvert^2 \leq \varepsilon \lVert 1/p \rVert_\infty \int_n^{n+1} p(t)\lvert f'(t) \rvert^2\,\mathrm d t + \Bigl(1+\frac{1}{\varepsilon}\Bigr) \int_n^{n+1} \lvert f(t) \rvert^2\,\mathrm d t.
		      \end{equation}
	\end{enumerate}
\end{lemma}

\begin{proof}
The estimates \eqref{Bischleben} and \eqref{50Watt} are proved in \cite[Lemma 4.1]{BST19}. 
For the convenience of the reader we verify
the estimate \eqref{pullover}, which is a variant of \cite[Lemma~9.32]{Te14}.
 Let $\varepsilon>0$ and $n\in\mathbb Z$. Then for $f\in D(T)$ and $x$,~$y\in[n,n+1]$
	\begin{equation*}
		\lvert f(x) \rvert^2 = \lvert f(y) \rvert^2 + 2\,\text{\rm Re}\int_y^x f'(t)\overline{f(t)}\,\mathrm d t.
	\end{equation*}
	By the mean value theorem we can choose $y$ such that $\lvert f(y) \rvert^2 = \int_n^{n+1} \lvert f(t) \rvert^2\,\mathrm d t$. Thus, by the Cauchy--Schwarz inequality and $2\alpha\beta\leq \alpha^2 + \beta^2$ for $\alpha$,~$\beta\in\mathbb R$ we obtain
	\begin{equation*}
		\begin{aligned}
			\lvert f(x) \rvert^2 & \leq \int_n^{n+1} \lvert f(t) \rvert^2\,\mathrm d t 
			\\[0.5\baselineskip]
			& \quad + 2\left(\frac{1}{\varepsilon}\int_{n}^{n+1}\lvert f(t) \rvert^2\,\mathrm d t\right)^{1/2} \cdot \left(\lVert 1/p \rVert_\infty \varepsilon\int_{n}^{n+1}p(t)\lvert f'(t) \rvert^2\,\mathrm d t \right)^{1/2} \\[0.5\baselineskip]
			             & \leq \varepsilon \lVert 1/p \rVert_\infty \int_n^{n+1} p(t)\lvert f'(t) \rvert^2\,\mathrm d t + \Bigl(1+\frac{1}{\varepsilon}\Bigr) \int_n^{n+1} \lvert f(t) \rvert^2\,\mathrm d t,
		\end{aligned}
	\end{equation*}
	which leads to \eqref{pullover}.
\end{proof}

For the proofs of Theorems~\ref{thm1} -- \ref{thm3}
it is no restriction to consider $f\in D_-(T)\setminus \{0\}$ and to assume 
that $q_-$ is positive on a set of positive Lebesgue measure.

\begin{proof}[Proof of Theorem~\ref{thm1}] 
  Let $1/p\in L^\infty(\mathbb R)$ and consider $\alpha$, $\beta$ as in \eqref{erdnuss1}. Choose $\varepsilon=(2 \lVert q_- \Vert_{\mathrm u}\lVert 1/p \rVert_\infty)^{-1}>0$. The estimate in \eqref{pullover} of Lemma~\ref{Chewbacca} yields
	\begin{equation}
		\label{gelb}
		\begin{split}
			\lVert q_-f^2 \rVert_1&=\int_{\mathbb R} q_-(t) \lvert f(t) \rvert^2\,\mathrm d t \\[0.5\baselineskip]
			&\leq \lVert q_- \Vert_{\mathrm u}\sum_{n\in\mathbb Z} \sup_{t\in [n,n+1]}\lvert f(t) \rvert^2\\[0.5\baselineskip]
			&\leq \lVert q_- \Vert_{\mathrm u} \left(\varepsilon \lVert 1/p \rVert_\infty\lVert \sqrt p f' \rVert_{2}^2 + \left(1+\frac{1}{\varepsilon}\right)\Vert f \rVert_{2}^2\right)\\[0.5\baselineskip]
			&=\frac{1}{2}\lVert \sqrt{p} f' \rVert_{2}^2  + \bigl(\lVert q_- \Vert_{\mathrm u}+ 2\lVert 1/p \rVert_\infty\lVert q_- \Vert_{\mathrm u}^2\bigr)\Vert f \rVert_{2}^2\\[0.5\baselineskip]
			&=\frac{1}{2}\lVert \sqrt{p} f' \rVert_{2}^2  + \frac{\alpha}{2}\Vert f \rVert_{2}^2.
		\end{split}
	\end{equation}
	Together with Lemma~\ref{A.6} we obtain
	\begin{equation*}
		\lVert \sqrt{p} f' \rVert_{2}^2 = 2\lVert \sqrt{p} f' \rVert_2^2 - \lVert \sqrt{p} f' \rVert_2^2 \leq 2 \lVert q_-f^2 \rVert_1 -\lVert \sqrt{p} f' \rVert_2^2 \leq \alpha  \Vert f \rVert_2^2.
	\end{equation*}
	With \eqref{50Watt} in Lemma~\ref{Chewbacca}	and \eqref{gelb} we see
	\begin{equation*}
		\Vert f \rVert_\infty^2 \leq 2\sqrt{\lVert 1/p \rVert_\infty \alpha }\Vert f \rVert_2^2 = \beta \Vert f \rVert_2^2 \quad\text{and}\quad  \lVert q_-f^2 \rVert_1 \leq \alpha \Vert f \rVert_2^2
	\end{equation*}
	and hence Lemma~\ref{elfo} leads to the statements in Theorem~\ref{thm1}.
\end{proof}

\begin{proof}[Proof of Theorem~\ref{thm3}]
Suppose that $1/p\in L^\eta(\mathbb R)$ and $q_-\in L^s(\mathbb R)$, where $\eta$,~$s\in[1,\infty)$ with $\eta+s>2$. Since $\eta+s>2$ we obtain
	\begin{equation*}
		2\eta s - \eta -s = \eta(s-1) + s(\eta-1)\geq s-1 + \eta-1>0.
	\end{equation*}
	Let $\alpha$ and $\beta$ as in \eqref{erdnuss10} and \eqref{olala}, respectively. From Hölder's inequality we obtain
	\begin{equation}
		\label{qMinusEstS}
		\begin{split}
			\lVert q_-f^2 \rVert_1 & \leq \Vert f \rVert_\infty ^{\frac{2}{s}} \int_{\mathbb R} \lvert q_-(t) \rvert \lvert f(t) \rvert^{\frac{2(s-1)}{s}}\,\mathrm d t 
			\\[0.5\baselineskip] 
			& \leq \Vert f \rVert_\infty ^{\frac{2}{s}} \left(\int_{\mathbb R}\lvert q_-(t) \rvert^{s}\,\mathrm d t \right)^{\frac{1}{s}}\left(\int_{\mathbb R} \lvert f(t) \rvert^2
			\,\mathrm d t\right)^{\frac{s-1}{s}}                                                                                                                                                                                                                   \\[0.5\baselineskip]
			& = \lVert q_- \Vert_{s} \Vert f \rVert_{\infty}^{\frac{2}{s}} \Vert f \rVert_{2}^{\frac{2(s-1)}{s}}.
		\end{split}
	\end{equation}
	Thus, together with Lemma~\ref{Chewbacca}~(i) and Lemma~\ref{A.6} we obtain
	\begin{equation*}
		\begin{aligned}
			\Vert f \rVert_\infty^2 & = \left(\frac{\Vert f \rVert_\infty^{\frac{2(2\eta -1)}{\eta}}}{\Vert f \rVert_\infty^{\frac{2}{s}}}\right)^{\frac{\eta s}{2\eta s-\eta -s}}
			\leq\left(\frac{\left(\frac{2\eta-1}{\eta}\right)^2 \lVert 1/p \rVert_\eta \lVert \sqrt p f' \rVert_2^2 \Vert f \rVert_2^\frac{2(\eta-1)}{\eta}}{\Vert f \rVert_\infty^{\frac{2}{s}}}\right)^{\frac{\eta s}{2\eta s-\eta -s}} \\[0.5\baselineskip]
			                    & \leq \left(\left(\frac{2\eta-1}{\eta}\right)^2 \lVert 1/p \rVert_\eta \lVert q_- \Vert_s\right)^{\frac{\eta s}{2\eta s-\eta -s}} \Vert f \rVert_2^2 = \beta \Vert f \rVert_2^2.
		\end{aligned}
	\end{equation*}
	The estimate from \eqref{qMinusEstS} yields
	\begin{equation*}
		\lVert q_-f^2 \rVert_{1} \leq \lVert q_- \Vert_{s} \beta^{\frac{1}{s}} \Vert f \rVert_2^2 = \alpha \Vert f \rVert_2^2.
	\end{equation*}
		
	Now consider the case $q_-\in L^\infty(\mathbb R)$. Choose $\alpha$ and $\beta$ as in \eqref{erdnuss10} and \eqref{olala}, respectively. Observe that
\begin{equation}
	\label{meer}
	\lVert q_-f^2 \rVert_1\leq \lVert q_- \Vert_\infty \Vert f \rVert_2^2 = \alpha \Vert f \rVert_2^2.
\end{equation}
Lemma~\ref{Chewbacca}~(i) in combination with Lemma~\ref{A.6} and \eqref{meer} leads to
\begin{equation*}
	\begin{split}
		\Vert f \rVert_\infty^2 & \leq \left(\left(\frac{2\eta-1}{\eta}\right)^{2} \lVert 1/p \rVert_\eta \lVert \sqrt p f' \rVert_2^2\right)^{\frac{\eta}{2\eta-1}} \Vert f \rVert_2^{\frac{2(\eta-1)}{2\eta-1}} \\[0.5\baselineskip]
						  & \leq \left(\left(\frac{2\eta-1}{\eta}\right)^{2} \lVert 1/p \rVert_\eta \lVert q_- \Vert_\infty\right)^{\frac{\eta}{2\eta-1}} \Vert f \rVert_2^2 = \beta \Vert f \rVert_2^2.
	\end{split}
\end{equation*}
Now Theorem~\ref{thm3} follows from Lemma~\ref{elfo}.
\end{proof}

\begin{proof}[Proof of Theorem~\ref{thm2}]
Consider first the case $1/p\in L^\infty(\mathbb R)$ and $q_-\in L^s(\mathbb R)$, where $s\in[1,\infty)$. Let $\alpha$ and $\beta$ be as in \eqref{erdnuss10}. Again Hölder's inequality yields \eqref{qMinusEstS}. Lemma~\ref{Chewbacca}~(ii), \eqref{qMinusEstS} and Lemma~\ref{A.6} imply
	\begin{equation*}
		\begin{split}
			\Vert f \rVert_\infty^2 & = \left(\frac{\Vert f \rVert_\infty^4}{\Vert f \rVert_\infty^{\frac{2}{s}}}\right)^{\frac{s}{2s-1}}
			\leq \left(\frac{4\lVert 1/p \rVert_\infty \lVert \sqrt p f' \rVert_2^2 \Vert f \rVert_2^2}{\Vert f \rVert_\infty^{\frac{2}{s}}}\right)^{\frac{s}{2s-1}} \\[0.5\baselineskip]
			                  & \leq \left(4\lVert 1/p \rVert_\infty \lVert q_- \Vert_s\right)^{\frac{s}{2s-1}} \Vert f \rVert_2^2 = \beta \Vert f \rVert_2^2.
		\end{split}
	\end{equation*}
	By applying this to the estimate in \eqref{qMinusEstS} we arrive at
	\begin{equation*}
		\lVert q_- f^2 \rVert_{1} \leq \lVert q_- \Vert_{s} \beta^{\frac{1}{s}} \Vert f \rVert_2^2 = \alpha \Vert f \rVert_2^2,
	\end{equation*}
and again the statements in Theorem~\ref{thm2} follow from Lemma~\ref{elfo}.

The assertion for $1/p$,~$q_-\in L^\infty(\mathbb R)$ follows in a similar way. 
Consider $\alpha$, $\beta$ in \eqref{erdnuss10}. As before \eqref{meer} holds.
Lemma~\ref{Chewbacca}~(ii) in combination with Lemma~\ref{A.6} and \eqref{meer} implies
\begin{equation*}
	\Vert f \rVert_\infty^2\leq 2\sqrt{\lVert 1/p \rVert_\infty} \lVert \sqrt p f' \rVert_2 \Vert f \rVert_2 \leq 2\sqrt{\lVert 1/p \rVert_\infty \lVert q_- \Vert_\infty} \Vert f \rVert_2^2 = \beta \Vert f \rVert_2^2.\qedhere
\end{equation*}
\end{proof}

\begin{proof}[Proof of Proposition~\ref{propi}]
 Let $f\in D_-(T)$. In the case $1/p$,~$q_-\in L^1(\mathbb R)$ Lemma~\ref{A.6} and Lemma~\ref{Chewbacca}~(i) yield
	\begin{equation*}
		\Vert f \rVert_\infty^2 \leq \lVert 1/p \rVert_1 \lVert \sqrt p f' \rVert_2^2 \leq \lVert 1/p \rVert_1 \lVert q_-f^2 \rVert_1 \leq \lVert 1/p \rVert_1 \lVert q_- \Vert_1 \Vert f \rVert_\infty^2.
	\end{equation*}
	If $\lVert 1/p \rVert_1\lVert q_- \Vert_1<1$, then $\Vert f \rVert_\infty=0$ and hence $D_-(T)=\{0\}$. 
	This implies $(Tf,f)_r\geq 0$ for all $f\in D(T)$ and hence $\inf(\sigma(T)) \geq 0$.
\end{proof}

\end{document}